\documentclass[12pt]{article}
\usepackage[T2A]{fontenc}
\usepackage[utf8]{inputenc} 
\usepackage[tbtags]{amsmath}
\usepackage{amsfonts,amssymb,mathrsfs,amscd,amsthm}
 
\usepackage{graphicx} 
\usepackage{bmpsize}

\def\thtext#1{
   \catcode`@=11
   \gdef\@thmcountersep{. #1}
   \catcode`@=12
}

\def\threst{
   \catcode`@=11
   \gdef\@thmcountersep{.}
   \catcode`@=12
}

\theoremstyle{plain}
\newtheorem{thm}{Theorem} 
\newtheorem{prop}{Proposition}[section] 
\newtheorem{cor}[prop]{Corollary}
\newtheorem{ass}[prop]{Assertion}
\newtheorem{lem}[prop]{Lemma}

\theoremstyle{definition}

\newtheorem{rk}[prop]{Remark}

\def\0{\emptyset}

\def\:{\colon}

\newcommand{\dl}{\delta}

\newcommand{\e}{\varepsilon}

\renewcommand{\l}{\lambda}

\newcommand{\om}{\omega}
\newcommand{\Om}{\Omega}
\renewcommand{\r}{\rho}

\renewcommand{\S}{\Sigma}

\newcommand{\cA}{\mathcal{A}}

\newcommand{\cS}{\mathcal{S}}
\newcommand{\cT}{\mathcal{T}}

\newcommand{\R}{\mathbb{R}}

\def\rom#1{{\em #1}} 

\begin{document}

\title{Fermat--Steiner Problem in the Metric Space of Compact Sets endowed with Hausdorff Distance}
\author{A.~Ivanov, A.~Tropin, A.~Tuzhilin}
\date{}							

\maketitle

\begin{abstract}
Fermat--Steiner problem consists in finding all points in a metric space $Y$ such that the sum of distances from each of them to the points from some fixed finite subset of $Y$ is minimal. This problem is investigated for the metric space $Y=H(X)$ of compact subsets of a metric space $X$, endowed with the Hausdorff distance. For the case of a proper metric space $X$ a description of all compacts $K\in H(X)$ which the minimum is attained at  is obtained. In particular, the Steiner minimal trees for three-element boundaries are described. We also construct an example of a regular triangle in $H(\R^2)$, such that all its shortest trees have no ``natural'' symmetry. 
\end{abstract}

\section{Introduction and Preliminaries}

The initial statement of the problem that is now referred as the \emph{Steiner Problem\/} probably belongs to P.~Fermat, who asked to find a point in the plane, such that the sum of distances from it to three fixed points is minimal. J.~Steiner generalized the problem considering an arbitrary finite subset of the plane or of the three-space. Nowadays the Steiner Problem usually stands for the general problem of finding a shortest  tree connecting a finite subset of a metric space, which was stated for the case of the plane by V.~Jarnik and M.~K\"ossler~\cite{JarnikKossler}.  By \emph{Fermat--Steiner problem\/} we mean the J.~Steiner type generalization of the P.~Fermat problem, namely, for a fixed finite subset $\cA=\{a_1,\ldots,a_n\}$ of a metric space $(Y,\r)$ find all the points $y\in Y$ minimizing the value $S_\cA(y):=\sum_i\r(y,a_i)$. The set of all such points we denote by $\Sigma(\mathcal{A})$. 

Put $d(y,\cA)=\bigl(\r(y,a_1),\ldots,\r(y,a_n)\bigr)$ and  $\Omega(\cA)=\bigl\{d(y,\cA): y\in\Sigma(\cA)\bigr\}$. For each $d\in\Om(\cA)$ we put  $\Sigma_d(\mathcal{A})=\bigl\{y\in\S(\cA):d(y,\cA)=d\bigr\}$. Thus, the set $\Sigma(\mathcal{A})$  is partitioned into the classes $\Sigma_d(\mathcal{A})$, $d\in \Omega(\cA)$. Notice that generally the set $\S(\cA)$ may be empty, but the following result can be easily obtained from standard compactness and continuity arguments. 

\begin{ass}[Existence of Steiner Compact in Proper Metric Space]\label{ass:existence} 
Let $Y$ be a proper metric space, then the set $\S(\cA)$, and hence, the set $\Om(\cA)$ is not empty for an arbitrary nonempty finite $\cA\subset Y$.
\end{ass}

In the present paper we consider the Fermat--Steiner problem in the space $Y=H(X)$ of nonempty compact subsets of a given metric space $X$, endowed with the Hausdorff distance~\cite{Hausdorff}, see also~\cite{Burago}.  Geometry of the space $H(X)$ is used in such important applications as patterns recognition and comparison, constructing of continuous deformations of one geometrical object into another, see for example~\cite{Memoli} and~\cite{MemoliDG}. 

In the present paper the following results are obtained.

Let $\cA$ be a finite subset of $Y=H(X)$. Each element from $\Sigma(\mathcal{A})$ we call \emph{Steiner compact}. Notice that in most examples of the boundaries the Steiner compacts are not uniquely defined even in $\Sigma_d(\mathcal{A})$. In each class $\Sigma_d(\mathcal{A})$  minimal and maximal Steiner compacts with respect to the natural order generated by inclusion are naturally defined. We prove that if a class $\Sigma_d(\mathcal{A})$ is not empty, then each Steiner compact $K\in \Sigma_d(\mathcal{A})$ contains some minimal Steiner compact $K_\l\in\Sigma_d(\mathcal{A})$. 

For the case of a proper metric space $X$ we show that the set $\Sigma(\mathcal{A})$ is not empty for any finite nonempty $\cA$. Maximal Steiner compact is unique in each $\Sigma_d(\mathcal{A})$ and is equal to the intersection of the corresponding closed balls centered at the boundary compact sets. This maximal compact is denoted by $K_d(\mathcal{A})$. Also in this case we describe the whole class $\Sigma_d(\mathcal{A})$: a compact set $K$ belongs to $\Sigma_d(\mathcal{A})$, if and only if the inclusions $K_{\lambda}\subset K\subset K_d(\mathcal{A})$ hold for some minimal Steiner compact $K_{\lambda}\in \Sigma_d(\mathcal{A})$. Thus, the problem of finding the Steiner compacts is reduced to the description of maximal and minimal Steiner compacts in each class $\Sigma_d(\mathcal{A})$. Notice that finding a vector $d\in\Om(\cA)$ is a nontrivial problem. 

As an example, we consider the boundary set $\mathcal{A}=\{A_1,A_2,A_3\}\subset H(\mathbb{R}^2)$, where each $A_i$ consists of the following pair of points: a vertex of the regular triangle, and its image under the rotation by the angle $\frac{\pi}{3}$ around the center $o$ of the triangle. For this example we completely describe the set of the Steiner compacts. It turns out that this symmetrical boundary possesses exactly three classes $\Sigma_d(\mathcal{A})$ that are transferred one into another under rotations by the angles $\frac{2\pi}{3}$ and $\frac{4\pi}{3}$ around $o$. The maximal Steiner compact in each class is a nonconvex curvilinear $4$-gon. Moreover, each class contains unique minimal Steiner compact, and this Steiner compact consists of two points. The length of Steiner minimal tree is less than  three radii of the circle containing $\cA$.

Let us pass to Hausdorff distance geometry. Let $(X,\rho)$ be a metric space. Recall that the set  $B_r^X(A):=\bigl\{x\in X\: \rho(x,A)\leq r\bigr\}$, where $\rho(x,A)=\inf_{a\in A}\rho(x,a)$, is called a  \emph{closed neighborhood of radius $r$} of a set $A\subset X$ in $(X,\rho)$.  Notice that if $A$ consists of a single point, then $B_r^X(A)$ is the closed ball of radius $r$ in $X$ centered at this point. Below we omit the reference to the space and write $B_r(A)$ instead of $B_r^X(A)$, providing it does not lead to misunderstandings. 

The \emph{Hausdorff distance\/} between subsets $A$ and $B$ of a metric space $(X,\rho)$ is defined as the following value:
$$
d_H(A,B)=\inf\bigl\{r : B_r^X(A)\supset B, B_r^X(B)\supset A\bigr\}.
$$ 
An equivalent definition is given by the expression
$$
d_H(A,B)=\max\bigl\{\sup_{a\in A}\inf_{b\in B}\rho(a,b),\sup_{b\in B}\inf_{a\in A}\rho(a,b)\bigr\},
$$
i.e.,
$$
d_H(A,B)=\max\bigl\{\sup_{a\in A}\rho(a,B),\sup_{b\in B}\rho(b,A)\bigr\}.
$$
It is well-known that the Hausdorff distance is a metric on the space $H(X)$ of compact subsets of  $X$, see for example~\cite{Burago}.

Consider an arbitrary metric space $(X,\rho)$. For any points $a$ and $b$ from $X$ we sometimes write $|ab|$ instead of $\rho(a,b)$. 

We need the following technical results.

\begin{ass}\label{ass:point_estimate}
Let  $A$ and $B$ be compact sets in a metric space $(X,\rho)$. Then for any point $a\in A$ there exists a point $b\in B$ such that $\rho(a,b)\leq d_H(A,B)$.
\end{ass}

\begin{proof}
Indeed,
$$
d_H(A,B) \geq \sup_{a\in A}\inf_{b\in B}\rho(a,b)\geq \inf_{b\in B}\rho(a,b),
$$
where the first inequality follows directly from the definition of the Hausdorff distance, and the second inequality is valid for any point $a\in A$ because of the supremum definition. Since $B$ is compact and the function $\rho(a,b)$ is continuous, then the infimum $\inf_{b\in B}\rho(a,b)$ is attained at some point $b$. For this $b$ the inequality $\rho(a,b)\leq d_H(A,B)$ holds. Assertion is proved.
\end{proof}

\begin{ass}\label{ass:balls}
Let $A$ and $B$ be compact sets in a metric space $(X,\rho)$, and let $d_H(A,B)=r$. Then $A\subset B_r^X(B)$ and $B\subset B_r^X(A)$.  
\end{ass}

\begin{proof}
It suffices to show that an arbitrary point $a\in A$ belongs to a closed neighborhood $B_r^X(B)$. Due to Assertion~\ref{ass:point_estimate}, there exists a point $b\in B$ such that $\rho(a,b)\leq d_H(A,B)=r$. But then $\rho(a,B)\leq r$.
\end{proof}

\begin{ass}\label{ass:ass_intermediate_compact}
Let compact sets $A$, $B$, $C$, and $D$ be such that $A\subset B\subset C$ and $d_H(A,D)=d_H(C,D)=d$. Then $d_H(B,D)\leq d$.
\end{ass}

\begin{proof}
The equality $d_H(A,D)=d$ and Assertion~\ref{ass:balls} imply that $D\subset B_d(A)$. The inclusion $A\subset B$ imply the inclusion $B_d(A)\subset B_d(B)$, and hence, $D\subset B_d(B)$. Similarly, the equality $d_H(C,D)=d$ implies the inclusion $C\subset B_d(D)$. But the inclusion $B\subset C$ implies that $B\subset B_d(D)$. Thus, $D\subset B_d(B)$ and $B\subset B_d(D)$, that implies the inequality desired.
\end{proof}

\begin{figure}[h]
\begin{center}
\begin{minipage}[h]{0.50\linewidth}
\includegraphics[width=1\linewidth]{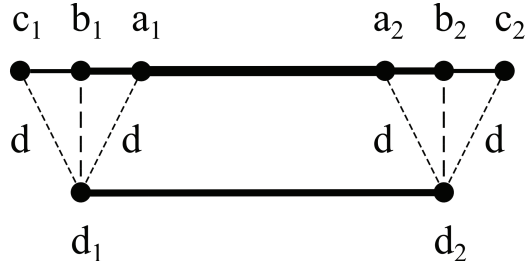}
\caption{An example to Assertion~\ref{ass:ass_intermediate_compact} with $d_H(B,D)<d$.}
\label{ris:counterexample_inclusion} 
\end{minipage}
\end{center}
\end{figure}

\begin{rk}\label{rk:rk_intermediate_compact}
An example of compact sets $A$, $B$, $C$, and $D$ from Assertion~\ref{ass:ass_intermediate_compact}, such that  $d_H(B,D)< d$ is shown in Figure~\ref{ris:counterexample_inclusion}. Here $A=[a_1,a_2]$, $B=[b_1,b_2]$, $C=[c_1,c_2]$, $D=[d_1,d_2]$ are segments such that $A\subset B\subset C$, $|a_1b_1|=|b_1c_1|=|a_2b_2|=|b_2c_2|$, the polygon $b_1d_1d_2b_2$ is a rectangle, and put $|a_1d_1|=d$. Then $d_H(A,D)=d_H(C,D)=|a_1d_1|=d$, but $d_H(B,D)=|b_1d_1|<d$.
\end{rk}

\begin{ass}\label{ass:ass_compact_neighbourhood}
Let $A$ be a compact subset in a proper space $X$. Then for any $d>0$ the set $B_d(A)$ is also a compact set.
\end{ass}

\begin{proof}
Since the ambient space $X$ is proper, then it suffices to show that $B_d(A)$is closed and bounded. Since $A$ is bounded, then we have $A \subset U_r(x)$ for some point $x$ from $X$ and some $r > 0$. Then  $B_d(A) \subset U_{r+d}(x)$, and hence, $B_d(A)$ is bounded. Now, let us show that $B_d(A)$ is closed.

Let $x$ be an adherent point of the set  $B_d(A)$. Then for any positive integer $i$ each $U_{1/i}(x)$ contains some point  $b_i$ from $B_d(A)$. Due to definition of $B_d(A)$, for each $i$ there exists  $a_i$ from $A$ such that $b_i$ belongs to  $B_d(a_i)$. Since $A$ is compact, then any sequence  $\{a_i\}$ of points from  $A$ contains a subsequence $\{a_{i_k}\}$ converging to some point $a\in A$. Consider the subsequence $\{b_{i_k}\}$ of the sequence $\{b_i\}$ corresponding to $\{a_{i_k}\}$. It converges to $x$, and hence, $|ax|\leq d$, due to continuity of the distance function. Therefore, $x\in B_d(A)$, and so, $B_d(A)$ is closed.
\end{proof}

\begin{figure}[h]
\begin{center}
\begin{minipage}[h]{0.50\linewidth}
\includegraphics[width=1\linewidth]{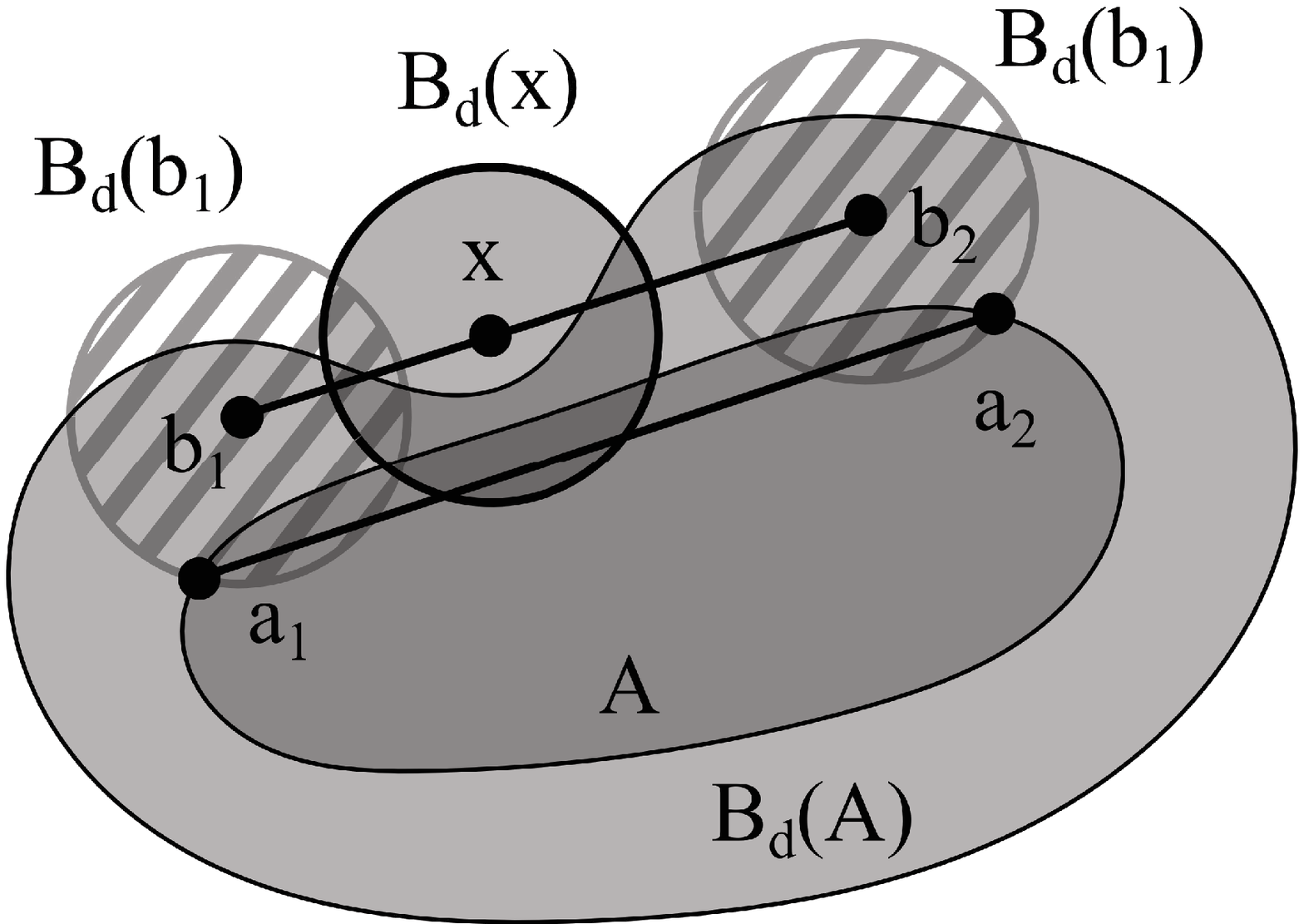}
\caption{To Assertion~\ref{ass:ass_convexity_neighbourhood}.}
\label{ris:convexity_neighbourhood} 
\end{minipage}
\end{center}
\end{figure}

\begin{ass}\label{ass:ass_convexity_neighbourhood}
Let  $A$ be a convex compact set in $\mathbb{R}^n$. Then the compact set $B_d(A)$ is also convex for any $d>0$.
\end{ass}

\begin{proof}
For any two points $b_1$, $b_2$ from $B_d(A)$ chose points $a_1,a_2$ from $A$ such that  $b_i\in B_d(a_i)$. Evidently, a closed $d$-neighborhood $B_d\bigl([a_1,a_2]\bigr)$ of a segment is convex. So, the segment $[b_1,b_2]$ is contained in  $B_d\bigl([a_1,a_2]\bigr)$. But $[a_1,a_2]\subset A$, therefore $B_d\bigl([a_1,a_2]\bigr)\subset B_d(A)$, and hence, $[b_1,b_2]$ is contained in $B_d(A)$.
\end{proof}

\begin{ass}\label{ass:boundedly_compact}
If a metric space $X$ is proper, then the space $H(X)$ of its compact subsets endowed with a Hausdorff metric is also proper.
\end{ass}

\begin{proof}
Consider an arbitrary bounded closed set $W\subset H(X)$. It lies in some ball $B_r^{H(X)}(K)\subset H(X)$ centered at some compact set $K\in H(X)$. Therefore, each compact set $C$ from $W$ is contained in a closed neighborhood $B=B_r^X(K)\subset X$ of the compact set $K$ in the space $X$. The set $B$ is closed, bounded and, hence, it is compact due to our assumptions.

It is well-known, see~\cite{Burago}, that the space $H(Y)$ is compact, if and only if the space $Y$ is compact itself. Therefore,  $C$, and hence, $W$ are contained in a compact set $H(B)$.  Thus, $W$ is a closed subset of a compact set, therefore  $W$ is compact. Assertion is proved. 
\end{proof}

The next result directly follows from Assertions~\ref{ass:existence} and~\ref{ass:boundedly_compact}. 

\begin{cor}\label{cor:H_space_existence}
Let $X$ be a proper metric space, and $\cA\subset H(X)$ be a nonempty finite subset. Then $\S(\cA)$ is not empty.
\end{cor}

\section{Structure of Steiner Compacts in Proper Metric Spaces}

We start with the case of an arbitrary metric space $X$.

\begin{ass}[Existence of Minimal Steiner Compacts]\label{ass:exist_minimal_compact} 
Let $X$ be a metric space, $\mathcal{A}=\{A_1,\dots,A_n\}\subset H(X)$, $\Sigma(\mathcal{A})\ne\emptyset$, and $d\in\Omega(\cA)$. Then each compact set from the class $\Sigma_d(\mathcal{A})$ contains at least one minimal Steiner compact from $\Sigma_d(\mathcal{A})$. 
\end{ass}

\begin{proof}
The set $\Sigma_d(\mathcal{A})$ is naturally ordered with respect to inclusion. Let us show that each chain $\{K_s: s\in \mathcal{S}\}$ from $\Sigma_d(\mathcal{A})$ possesses a lower bound, and apply Zorn's Lemma. 

\begin{lem}\label{lem:inter_chain}
The intersection $\widetilde{K}:=\cap_{s\in \mathcal{S}}K_s$ of all the elements of an arbitrary chain is a nonempty compact set. 
\end{lem}

\begin{proof}
Compact subsets of any metric (and hence, Hausdorff) space are closed, therefore the intersection of an arbitrary family of compact subsets of a metric space is compact (because it is a closed subset of a compact set). So, it remains to verify that  $\widetilde{K}$ is nonempty.

Choose an index $s_0$ and consider the index subset $\mathcal{S'}:=\{s\in \mathcal{S}: K_s\subset K_{s_0}\}$. In accordance with definition of a chain,  $\widetilde{K}=\cap_{s\in \mathcal{S'}}K_s$ for any $s_0$. The family $\{K_s: s\in \mathcal{S'}\}$ consists of closed subsets of the compact set $K_{s_0}$ satisfying the finite intersection property, therefore its intersection is nonempty, see for example~\cite{AlMir}. 
\end{proof}

\begin{lem}\label{lem:lim_chain}
For any  $\varepsilon>0$ there exists an element $K_{s''}$ of the chain, such that for any element $K_{s}$ of the chain, that is contained in $K_{s''}$, the inequality $d_H(K_s,\widetilde{K})<\varepsilon$ is valid.
\end{lem}

\begin{proof}
Assume the contrary. Then $d=\inf_{s} d_H(K_s,\widetilde{K})>0$. Indeed, if $d=0$, then for any $\varepsilon>0$ there exists a compact set $K_{s''}$ such that $d_H(K_{s''},\widetilde{K})<\varepsilon$,  but then the inequality $d_H(K_s,\widetilde{K})<\varepsilon$ is valid for all $K_s\subset K_{s''}$ also, a contradiction.

Fix some $s_0\in \cS$ and consider $\mathcal{S'}=\{s\in \mathcal{S}\: K_s\subset K_{s_0}\}$ again.  Open sets $\{X\setminus K_s: s\in \mathcal{S'}\}$ together with an arbitrary open neighborhood $U_{\e}(\widetilde{K})$ of $\widetilde{K}$, $\e>0$, form an open covering of the compact set $K_{s_0}$. Indeed, if it is not so, then there exists a point $x\in K_{s_0}$ that does not belong to any set $X\setminus K_s$, $s\in\cS'$, and also does not lie in $U_\varepsilon(\widetilde{K})$. But in this case $x$ belongs to all $K_s$, $s\in\cS'$, but $s\not\in\widetilde{K}=\cap_{s\in\cS'}K_s$, a contradiction. 

Take $\e=d/2$, and choose a finite subcovering from the corresponding open covering. Let $X\setminus K_{s'}$ be the greatest element with respect to inclusion among the elements of the form $X\setminus K_{s}$ from this covering. Then  $X\setminus K_{s'}$ and $U_{d/2}(\widetilde{K})$ cover the compact set $K_{s_0}$. But since $\widetilde{K}\subset K_{s'}$, then $d\leq d_H(K_{s'},\widetilde{K})=\inf\{r\: K_{s'}\subset U_r(\widetilde{K})\}$. So, $K_{s'}$ is not contained in $U_{d/2}(\widetilde{K})$, and hence the sets $X\setminus K_{s'}$ and $U_{d/2}(\widetilde{K})$ do not cover $K_{s'}$, and therefore, they do not cover  $K_{s_0}\supset K_{s'}$. This contradiction completes the proof of Lemma.
\end{proof}

For any bounded compact $A_j$ and any  $s\in\cS$, the triangle inequality implies that $|d_H(A_j,\widetilde{K})-d_H(A_j,K_s)|\leq d_H(K_s,\widetilde{K})$, where $d_H(A_j,K_s)=d_j$, because $K_s\in\Sigma_d(\cA)$. Due to Lemma~\ref{lem:lim_chain}, for any $\varepsilon>0$ there exists an element $K_{s'}$ of the chain, such that for any element $K_{s}$ of the chain that is contained in $K_{s'}$, the inequality $d_H(K_s,\widetilde{K})<\varepsilon$ is valid. Then $\bigl|d_H(A_j,\widetilde{K})-d_j\bigr|<\varepsilon$ and, due to arbitrariness of $\varepsilon>0$, the equality $d_H(A_j,\widetilde{K})=d_j$ holds. Therefore, the compact set $\widetilde{K}$ also belongs to $\Sigma_d(\mathcal{A})$. Due to construction, it is a lower bound of the chain under consideration. Therefore, due to Zorn's Lemma, any compact set $K$ from $\Sigma_d(\mathcal{A})$ majorizes some minimal element. In our terms the latter means that any $K$ from $\Sigma_d(\mathcal{A})$ contains at least one minimal Steiner compact.
\end{proof}

\begin{ass}[Intermediate  Steiner Compacts]\label{ass:ass_intermediate_1_compact}
Let $X$ be a metric space, $\mathcal{A}=\{A_1,\dots,A_n\}\subset H(X)$, $\Sigma(\mathcal{A})\ne\emptyset$, and $d\in\Omega(\cA)$. Assume that $K_1,\,K_2\in \Sigma_d(\mathcal{A})$, such that $K_1\subset K_2$. Then any compact set $K$ such that $K_1\subset K\subset K_2$ belongs to $\Sigma_d(\mathcal{A})$ too. 
\end{ass}

\begin{proof}
Due to Assertion~\ref{ass:ass_intermediate_compact}, for any compact set $K$ such that $K_1\subset K\subset K_2$ the inequalities $d_H(A_i,K)\leq d_H(A_i,K_1)$ are valid for all  $i=1,\dots,n$. Therefore, $S_\cA(K)\leq S_\cA(K_1)$ (recall that by $S_\cA(K)$ we denote the value $\sum_i d_H(K,A_i)$). But the function  $S_\cA$ attains it least value at the compact set $K_1$, therefore $S_\cA(K)= S_\cA(K_1)$, and hence,  $d_H(A_i,K)= d_H(A_i,K_1)$ for all $i$. So, $K\in\Sigma_d(\mathcal{A})$. 
\end{proof}

Let $X$ be a proper metric space. Let $d=(d_1,\dots,d_n)$ be a vector with nonnegative components. Put $K_d(\mathcal{A}):= \cap_{i=1}^{n} B_{d_i}^X(A_i)$. Due to Assertion~\ref{ass:ass_compact_neighbourhood}, $K_d(\cA)$  is a compact set (it may be empty).  

The following Assertion generalizes results from~\cite{Schlicker2007} describing compact sets located between two fixed compact sets in a given distances in the sense of Hausdorff (so-called \emph{compact sets in $s$-location\/}) to the case of three and more compact sets.

\begin{ass}[Existence and Uniqueness of Maximal Steiner Compact]\label{ass:uniqueness_maximal_compact} 
Let $X$ be a proper metric space, $\mathcal{A}=\{A_1,\dots,A_n\}\subset H(X)$. Then $\S(\cA)$ and $\Om(\cA)$ are not empty, and for any $d\in\Omega(\cA)$ the set $\Sigma_d(\mathcal{A})$ contains  unique maximal Steiner  compact, and this Steiner compact is equal to $K_d(\mathcal{A})$. 
\end{ass}

\begin{proof}
The sets $\S(\cA)$ and $\Om(\cA)$ are not empty due to Corollary~\ref{cor:H_space_existence}. Let $d$ be an arbitrary element of $\Om(\cA)$, and let $K$ be an arbitrary Steiner compact from  $\Sigma_d(\mathcal{A})$. Recall that $d=(d_1,\dots,d_n)$, where $d_i=d_H(A_i,K)$.  Due to Assertion~\ref{ass:balls}, $K\subset B_{d_i}^X(A_i)$ for all $i$, and hence, $K\subset \cap_i B_{d_i}^X(A_i)=K_d(\cA)$.

Put $d'_i=d_H(A_i,K_d(\mathcal{A}))$. Show that  $d'_i\leq d_i$. To do that it suffices to verify that $K_d(\mathcal{A})\subset B_{d_i}^X(A_i)$ and $A_i\subset B_{d_i}^X(K_d(\mathcal{A}))$. The first inclusion is valid in accordance with the definition of  $K_d(\mathcal{A})$. The second inclusion is valid, because $K\subset K_d(\mathcal{A})$, and hence, $B_{d_i}^X(K)\subset B_{d_i}^X(K_d(\mathcal{A}))$, and $A_i\subset B_{d_i}^X(K)$ in accordance with Assertion~\ref{ass:balls}.

Thus, $S_\cA\bigl(K_d(\mathcal{A})\bigr)\leq S_\cA(K)$. But the function $S_\cA$ attains its least value at $K$, therefore $S_\cA\bigl(K_d(\mathcal{A})\bigr)= S_\cA(K)$, and hence, $d'_i=d_i$ for any $i$, and so $K_d(\mathcal{A})\in\Sigma_d(\mathcal{A})$. The inclusion $K\subset K_d(\mathcal{A})$ proved above implies that $K_d(\mathcal{A})$ is the greatest element with respect to inclusion in the  class $\Sigma_d(\mathcal{A})$. Therefore, it is unique maximal element in $\Sigma_d(\mathcal{A})$. 
\end{proof}

\begin{thm}[Structure of $\Sigma_d(\mathcal{A})$ in a Proper Metric Space]\label{th:th_structure_1_compact}
Let $X$ be a proper metric space, $\mathcal{A}=\{A_1,\dots,A_n\}\subset H(X)$. Then $\S(\cA)$ and $\Om(\cA)$ are not empty, and for any $d\in\Omega(\cA)$ a compact set $K$ belongs to the class $\Sigma_d(\mathcal{A})$, if and only if $K_{\lambda}\subset K\subset K_d(\mathcal{A})$ for some minimal Steiner compact $K_{\lambda}\in \Sigma_d(\mathcal{A})$ and the unique maximal Steiner compact $K_d(\mathcal{A})$ from $\Sigma_d(\mathcal{A})$.
\end{thm}

\begin{proof}
The sets $\S(\cA)$ and $\Om(\cA)$ are not empty due to Corollary~\ref{cor:H_space_existence}. The existence, uniqueness and form of a maximal Steiner compact in $\Sigma_d(\mathcal{A})$ follow from Assertion~\ref{ass:uniqueness_maximal_compact}. Due to Assertion~\ref{ass:exist_minimal_compact}, each Steiner  compact $K$ contains some minimal Steiner compact $K_{\lambda}$. Thus, for any Steiner compact $K$ the inclusions $K_{\lambda}\subset K\subset K_d(\mathcal{A})$ are valid.

Conversely, due to Assertion~\ref{ass:ass_intermediate_1_compact}, the inclusions $K_{\lambda}\subset K\subset K_d(\mathcal{A})$ imply that $K\in\Sigma_d(\mathcal{A})$. 
\end{proof}

Consider the particular case  $X=\mathbb{R}^m$ in more details.

\begin{cor}[]\label{cor:cor_continuum_1_compact} 
Let  $\mathcal{A}=\{A_1,\dots,A_n\}\subset H(\R^m)$, and $d\in\Omega(\cA)$. If the maximal Steiner compact $K_d(\mathcal{A})$ from the class $\Sigma_d(\mathcal{A})$ is convex and does not coincide with some\/ {\rm(}and hence, with any\/{\rm)} minimal Steiner compact $K_\l\in\Sigma_d(\mathcal{A})$, then the cardinality of $\Sigma_d(\mathcal{A})$ is continuum.
\end{cor}

\begin{proof}
Consider an arbitrary point $x$ from $K_d(\mathcal{A})$ and some point $y\in K_d(\mathcal{A})\setminus K_\l$ that exists due to our assumptions. Then the segment $[x,y]$ belongs to $K_d(\mathcal{A})$ due to its convexity. Further, since $K_\l$ is a closed subset of $K_d(\mathcal{A})$, then there exists an open ball $U$ centered at the point $y$ that does not intersect $K_\l$. Therefore, the interval $[x,y]\cap U$ does not intersect $K_\l$, it is contained in  $K_d(\mathcal{A})$, and its cardinality is continuum.  Thus, for any point $z\in [x,y]\cap U$, the compact set $K(z)=K_\l\cup\{z\}$ satisfies the inclusions $K_\l\subset K(z) \subset K_d(\mathcal{A})$, and hence, in accordance with Thorem~\ref{th:th_structure_1_compact}, each $K(z)$ belongs to $\Sigma_d(\mathcal{A})$. Thus, $\Sigma_d(\mathcal{A})$ contains a subset of cardinality continuum.

It remains to notice, that the cardinality of the set of all compact subsets of $\R^m$ is continuum also. Indeed, it is well-known, see for example~\cite{Shilov}, that the cardinality of the set of all real sequences is continuum. So, the same is valid for the family of all the sequences of points in  $\R^m$. On the other hand, each compact set in $\R^m$ is a closure of some its at most countable subset that can be considered as a sequence of points in $\R^m$. 
\end{proof}

\begin{ass}[Convexity of Maximal Steiner compact for a Convex Boundary]\label{ass:convexity_1_compact} 
Let  $\mathcal{A}=\{A_1,\dots,A_n\}\subset H(\R^m)$, and let all $A_i$ be convex. Then for any $d\in\Omega(\cA)$ the maximal Steiner compact $K_d(\mathcal{A})$ is also convex.
\end{ass}

\begin{proof}
Due to Assertion~\ref{ass:uniqueness_maximal_compact}, $K_d(\mathcal{A}) = \cap_{i=1}^{n} B_{d_i}(A_i)$. Since  $A_i$ are convex, then in accordance with Assertion~\ref{ass:ass_convexity_neighbourhood}, the sets $B_{d_i}(A_i)$ are convex for all $d_i$, and so, their intersection $K_d(\mathcal{A})$ is also convex.
\end{proof}

\section{Example of Symmetric Boundary in $H(\R^2)$ with Three Classes of Steiner Compacts}

In this Section we consider an example of three-element boundary in $H(\R^2)$ (an equilateral triangle). For it we construct explicitly all minimal and maximal Steiner compacts (and, hence, all Steiner compacts). In spite of the symmetry of the boundary with respect to the rotations of the plane by the angles $\pm2\pi/3$, the corresponding solutions to Fermat--Steiner problem turn out to be not invariant with respect to those rotations. 

In what follows we write $|xA|$ instead of $\inf_{a\in A}|xa|$ for any point  $x\in \mathbb{R}^2$ and any compact set $A\subset \mathbb{R}^2$. By $E(ab,s)$ we denote the ellipse with foci are located at  the points  $a$, and $b$ and and whose sum of focal radii is equal to $s$. 

Consider a boundary $\mathcal{A}=\{A_1,A_2,A_3\}\subset H(\mathbb{R}^2)$, where  $A_i=\{a_i,b_i\}$, $i=1,\,2,\,3$, the points $a_i$ are located at the vertices of a regular triangle inscribed in the unit circle centered at the origin $o$, and the points $b_i$ are obtained from the corresponding points $a_i$ by the rotation with respect to $o$ by the angle $\pi/3$.

\begin{figure}[h]
\begin{center}
\includegraphics[width=0.55\linewidth]{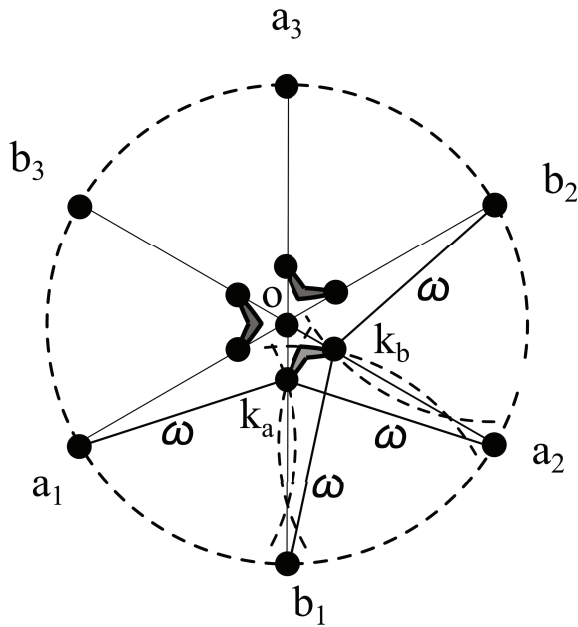}
\caption{Boundary compacts $A_i=\{a_i,b_i\}$ and three maximal Steiner compacts (filled in gray).}
\label{ris:splitting_arc_60_common}
\end{center}
\end{figure}

Let $\omega\in(\sqrt3/2,1)$. By $k_a(\om)$ and $k_b(\om)$ we denote the closest to $o$ points from the sets  $\partial B_{\omega}(a_1)\cap \partial B_{\omega}(a_2)$ and $\partial B_{\omega}(b_1)\cap \partial B_{\omega}(b_2)$, respectively (see Figure~\ref{ris:splitting_arc_60_common}). For the chosen $\om$ the points $k_a(\om)$ and $k_b(\om)$ belong to the segments $[o,b_1]$ and $[o,a_2]$, respectively. Put $t=\bigl|ok_a(\om)\bigr|=\bigl|ok_b(\om)\bigr|$. It is clear that $t\in(0,1/2)$, and $\om=\sqrt{1+t^2-t}$. In what follows it is convenient to use the parameter $t$. Consider the family $\cT$ of two-point compacts $T_{ab}(t)=\bigl\{k_a(t),k_b(t)\bigr\}$, $t\in(0,1/2)$, where $k_a(t)=k_a\bigl(\om(t)\bigr)$, and $k_b(t)=k_b\bigl(\om(t)\bigr)$.  To be short, we sometimes omit the explicit parameter and write $k_a$ instead of $k_a(t)$, etc.

Put $t_0=\frac{\sqrt{5}}{4}-\frac{1}{2}\sqrt{\sqrt{5}-7/4}$, $\om_0=\sqrt{1+t_0^2-t_0}$,  and $T_{ab}(t_0)=K_0$. The value  $t_0$ is found as a root of degree $4$ algebraic equation, see Lemma~\ref{lem:k_a_k_b_length} below. The following result solves completely the Steiner Problem for the chosen boundary $\cA$ in $H(\R^2)$. 

\begin{thm}[]\label{th:th_splitting_arc} 
Under the above notations,  
\begin{enumerate}
\item\label{item:values} The compact  $K_0$ is a solution to the Fermat--Steiner problem for $\cA$, and the distances from $K_0$ to $A_1$ and $A_2$ are equal to $\om_0$, and the distance from $K_0$ to $A_3$ is equal to  $\sqrt{1+t_0^2+t_0}$, i.e., $K_0\in\Sigma_{d}(\mathcal{A})$, where $d=(\om_0,\om_0,\sqrt{1+t_0^2+t_0})$\rom; in particular, 
\begin{multline*}
S_\cA(K_0)=2\om_0+\sqrt{1+t_0^2+t_0}=\\ 
=\sqrt{\frac{35}{8}+\frac{3}{8}\sqrt{-47 + 80 \sqrt{5}}}=2.94645\dots<3;
\end{multline*}
\item\label{item:min_com} $K_0$ is unique minimal Steiner compact in $\Sigma_{d}(\mathcal{A})$\rom;
\item\label{item:max_com} The maximal Steiner compact from $\Sigma_{d}(\mathcal{A})$ is equal to the corresponding set $K_d(\cA)$\rom; 
\item The set $\Sigma(\mathcal{A})$ of all Steiner compacts is partitioned into three classes $\Sigma_{g}(\mathcal{A})$, the corresponding vectors $g$ are obtained from $d$ by cyclic permutations, and the corresponding $\Sigma_{g}(\mathcal{A})$ are obtained from  $\Sigma_{d}(\mathcal{A})$ by plane rotation by the angles $2\pi/3$ and $4\pi/3$ around the point  $o$.
\end{enumerate}
\end{thm}

\begin{proof}
We start with the following technical Lemma.

\begin{lem}\label{lem:k_a_k_b_length}
Under the above notations, 
\begin{gather*}
d_H\bigl(A_1,T_{ab}(t)\bigr)=|a_1k_a|=|b_1k_b|=\om, \qquad d_H\bigl(A_2,T_{ab}(t)\bigr)=|a_2k_a|=|b_2k_b|=\om, \\
d_H\bigl(A_3,T_{ab}(t)\bigr)=|a_3k_b|=|b_3k_a|=\sqrt{1+t^2+t},
\end{gather*}
therefore, 
$$
S_\cA\bigl(T_{ab}(t)\bigr)=\bigl|b_1k_b(t)\bigr|+\bigl|b_2k_b(t)\bigr|+\bigl|a_3k_b(t)\bigr|=2\sqrt{1+t^2-t}+\sqrt{1+t^2+t}.
$$ 
The function $f(t)=S_\cA\bigl(T_{ab}(t)\bigr)$ attains its least value at the unique point $t_0$, where $t_0=\frac{\sqrt{5}}{4}-\frac{1}{2}\sqrt{\sqrt{5}-7/4}=0.210424\dots$, and this least value is equal to $\sqrt{\frac{35}{8}+\frac{3}{8}\sqrt{-47 + 80 \sqrt{5}}}=2.94645\dots$.
\end{lem}

\begin{proof}
The Hausdorff distances between two-element sets $A_i$ and $T_{ab}(t)$, $i=1,\,2,\,3$, can be calculated directly taking into account that $\om>1-t$. The value $S_\cA\bigl(T_{ab}(t)\bigr)$ is equal to the sum of those distances. The corresponding function $f(t)$ is defined and differentiable for all $t\in\R$, and its derivative has the form
$$
f'(t)=\frac{2t-1}{\sqrt{t^2-t+1}}+\frac{2 t+1}{2 \sqrt{t^2+t+1}}.
$$ 
The equation $f'(t)=0$ is equivalent to
$$
2 \sqrt{t^2+t+1} (1-2 t)=(2 t+1) \sqrt{t^2-t+1}.
$$ 
Its solutions have to satisfy the inequality $-1/2\le t\le 1/2$, and for such $t$ the equation is equivalent to the following equation of degree $4$:
$$
4 t^4+t^2-5 t+1=\frac{1}{4}(4 t^2-2 \sqrt{5} t-\sqrt{5}+3)(4 t^2+2 \sqrt{5} t+\sqrt{5}+3)=0.
$$
The latter equation possesses two real roots (the roots of the first quadratic term), and only one of them belongs to the interval $[-1/2,1/2]$. This root is $t_0=\frac{\sqrt{5}}{4}-\frac{1}{2}\sqrt{\sqrt{5}-7/4}$. The point $t_0$ is the unique minimum point of the function $f$. It remains to calculate $f(t_0)$. Lemma is proved\footnote{The calculations in the proof of Lemma were proceeded with {\sl Mathematica\/} by Wolfram Research Inc.}. 
\end{proof}

Let us return to the proof of Theorem. Due to Corollary~\ref{cor:H_space_existence}, for the boundary $\mathcal{A}\subset H(\mathbb{R}^2)$ a Steiner compact $K$ does exist. Due to definition of Steiner compact and in accordance with Lemma~\ref{lem:k_a_k_b_length}, we have $S_\cA(K)\le S_\cA(K_0)<3$. Put $\dl_i=d_H(A_i,K)$. Without loss of generality, assume that the values  $\dl_i$ are ordered as $\dl_1\leq \dl_2 \leq \dl_3$, and put $\dl_1+\dl_2 = s$ and $\dl=(\dl_1,\dl_2,\dl_3)$. Then $\dl_1 \leq s/2$, and $\dl_2\ge s/2$. Besides, $s<2$, because otherwise $S(K)=s+\dl_3\geq s+\dl_2\ge 3 s/2 \ge 3$, that contradicts to the estimate obtained in Lemma~\ref{lem:k_a_k_b_length}.

Put $P_{s,\dl_1}=B_{\dl_1}(A_1)\cap B_{\dl_2}(A_2)$,  $M_{a,s,\dl_1}=B_{\dl_1}(a_1)\cap B_{\dl_2}(a_2)$, and $M_{b,s,\dl_1}=B_{\dl_1}(b_1)\cap B_{\dl_2}(b_2)$, where $\dl_2=s-\dl_1$, see Figure~\ref{ris:splitting_arc_60_1}. Since $|a_1b_2|=2$, and $\dl_1+\dl_2<2$, then $B_{\dl_1}(a_1)\cap B_{\dl_2}(b_2)=\0$, and hence, it is not difficult to verify (using set-theoretical formulas) that $P_{s,\dl_1} \cap B_{\dl_1}(a_1) = M_{a,s,\dl_1}$ and $P_{s,\dl_1} \cap B_{\dl_2}(b_2) = M_{b,s,\dl_1}$. On the other hand, $K\subset K_\dl(\cA)\subset P_{s,\dl_1}$ in accordance to Assertion~\ref{ass:uniqueness_maximal_compact}. Besides, since $B_{\dl_i}(K)\supset\{a_i,b_i\}$, $i=1,\,2,\,3$, then the compact set $K$ intersects all $B_{\dl_i}(a_i)$ and $B_{\dl_i}(b_i)$. Therefore, the sets $M_{a,s,\dl_1} \cap K$ and $M_{b,s,\dl_1} \cap K$ are nonempty.

Notice that $E_{b,s}:=\cup_{\dl_1\in[0,s]}M_{b,s,\dl_1}$ is a compact set bounded by the ellipse $E(b_1b_2, s)$ with foci at $b_1$, $b_2$ and the sum of focal radii $s$. By $E^{-}_{b,s}$ we denote the set $\cup_{\dl_1\in[0,s/2]}M_{b,s,\dl_1}$ (see Figure~\ref{ris:splitting_arc_60_2}). Since the set $M_{b,s,\dl_1} \cap K$ is nonempty for some $\dl_1\le s/2$, then $K\cap E^{-}_{b,s}\ne\0$. 

\begin{figure}[h]
\begin{center}
\begin{minipage}[h]{0.43\linewidth}
\includegraphics[width=1\linewidth]{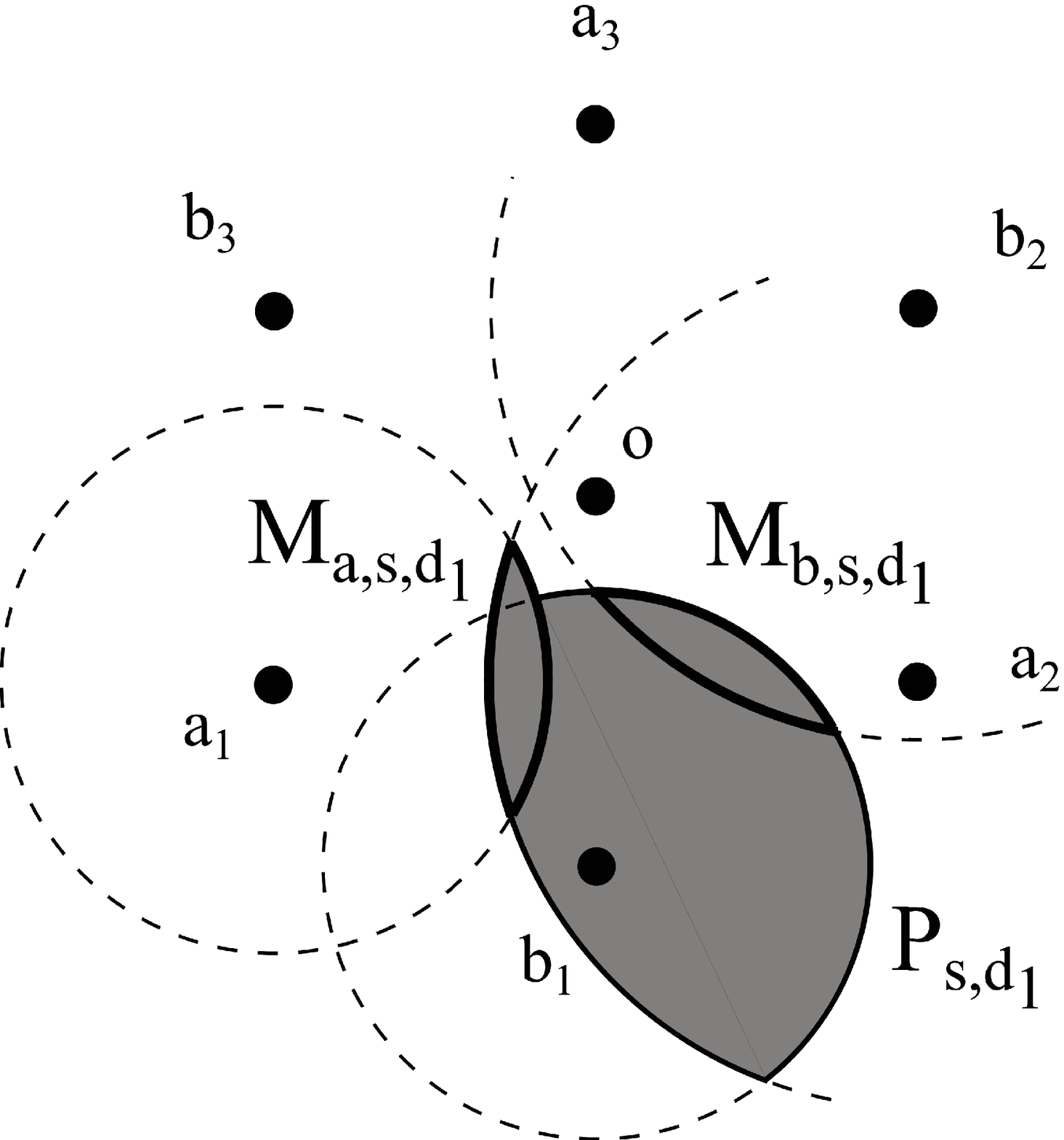}
\caption{The compact sets $M_{a,s,\dl_1}$, $M_{b,s,\dl_1}$, and $P_{s,\dl_1}$ (filled with gray).}
\label{ris:splitting_arc_60_1}
\end{minipage}
\hfill
\begin{minipage}[h]{0.47\linewidth}
\includegraphics[width=1\linewidth]{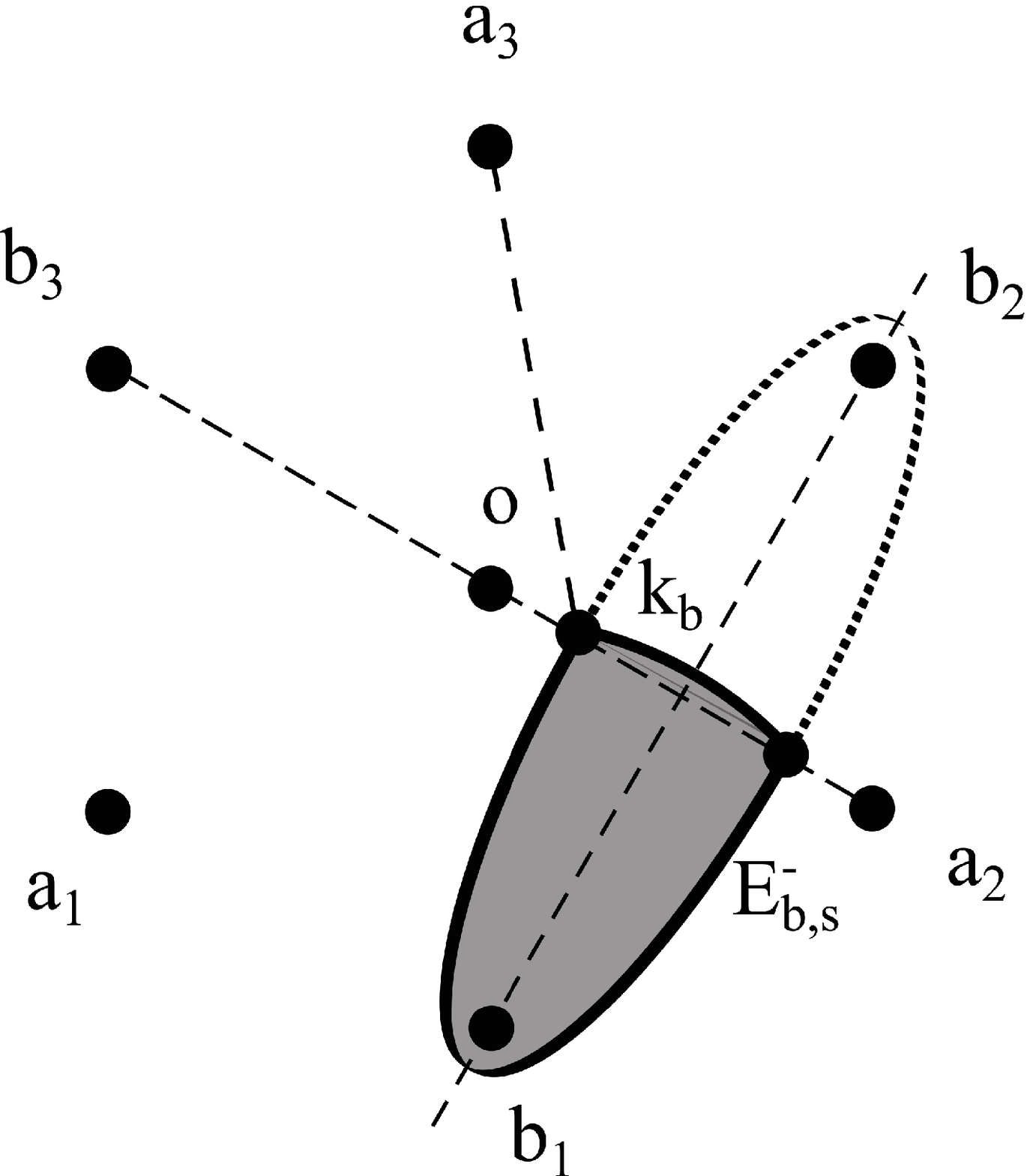}
\caption{The ellipse $E(b_1b_2, s)$, and the compact set $E^{-}_{b,s}$ (filled with gray).}
\label{ris:splitting_arc_60_2}
\end{minipage}
\end{center}
\end{figure}

By $q$ we denote the intersection point of the ellipse $E(b_1b_2,s)$ and the segment $[o,a_2]$, which is closest to $o$. Since $q\in E(b_1b_2,s)$, then $s=|b_1q|+|b_2q|$. It is clear that $|b_3q|=|a_3o|+|oq|>|a_3q|$, therefore $|a_3E_{b,s}^-|=|a_3q|<|b_3q|=|b_3E_{b,s}^-|$. Further, since $K\cap E^{-}_{b,s}\ne\0$, then we have $d_H(K,A_3)=\dl_3\geq |a_3E_{b,s}^-|=|a_3q|$. Thus, $S_\cA(K)=s+\dl_3\geq |b_1q|+|b_2q|+|a_3q|$.

Put $|oq|=t_q$. Notice that  $t_q\in(0,1/2)$, and hence, under the notations from Lemma~\ref{lem:k_a_k_b_length}, $q=k_b(t_q)$ and $|b_1q|+|b_2q|+|a_3q|=S_\cA\bigl(T_{ab}(t_q)\bigr)$. Therefore, $S_\cA(K)\ge S_\cA\bigl(T_{ab}(t_q)\bigr)$.  On the other hand, $S_\cA(K)\le S_\cA\bigl(T_{ab}(t_q)\bigr)$ because $K$ is a Steiner compact, so  $S_\cA(K)=S_\cA\bigl(T_{ab}(t_q)\bigr)$, and hence, for some $t_q$, the two-point compact set $T_{ab}(t_q)$ from the family  $\cT$ is a Steiner compact also. 

Thus, the family $\cT$ of two-point compact sets contains a Steiner compact. Therefore, the two-point compact set  $K_0=T_{ab}(t_0)$ which the function $S_\cA$ restricted to $\cT$ attains its least value at due to Lemma~\ref{lem:k_a_k_b_length}, is a Steiner compact. The latter implies that $S_\cA(K)=S_\cA\bigl(T_{ab}(t_0)\bigr)$. The Hausdorff distances between  $K_0$ and $A_i$ are calculated in Lemma~\ref{lem:k_a_k_b_length}, that completes the proof of Item~(\ref{item:values}) of Theorem.

Now let us describe the set $\Sigma(\cA)$ of all Steiner compacts. Notice that in accordance with Lemma~\ref{lem:k_a_k_b_length}, the function $f(t)=S_\cA\bigl(T_{ab}(t)\bigr)$ attends its least value at the unique point $t_0$, therefore, $t_q=t_0$, and $K_0$ is the unique Steiner compact in $\cT$. So, for an arbitrary Steiner compact $K$ from any class $\Sigma_{\dl}(\mathcal{A})$, $\dl_1\leq \dl_2 \leq \dl_3$, the relations $S_\cA(K)=s_0+\dl_3 = |b_1q|+|b_2q|+|a_3q|=S_\cA(K_0)$ are valid, where $q=k_b(t_0)$, $s_0=\dl_1+\dl_2=|b_1q|+|b_2q|$, and $\dl_3=d_H(K,A_3)=|a_3q|$. Also notice, that $q=k_b(t_0)$ is the closest point to $a_3$ from the half-ellipse $E^{-}_{b,s_0}$, therefore $K\cap E^{-}_{b,s_0}=\bigl\{k_b(t_0)\bigr\}$, because $d_H(K,A_3)>|a_3q|$ otherwise. Further, since $M_{b,s_0,\dl_1} \subset E^{-}_{b,s_0}$ and $M_{b,s_0,\dl_1} \cap K\ne\emptyset$, then $M_{b,s_0,\dl_1} \cap K=\bigl\{k_b(t_0)\bigr\}$.

The Steiner compact $K_0$ constructed above belongs to the class $\Sigma_{d}(\mathcal{A})$, where  $d_1=d_2=s_0/2$, and $d_3=\bigl|a_3k_b(t_0)\bigr|$. Assume that there exists another class $\Sigma_{g}(\mathcal{A})$, $g=(g_1,g_2,g_3)$, $g_1\le g_2\le g_3$.  Then, as it has been already shown above, $g_1+g_2=s_0$, and $g_3=\bigl|a_3k_b(t_0)\bigr|$. If $g\neq d$, then $g_1<g_2$. But then $M_{b,s_0,g_1}$ does not contain the point $k_b(t_0)$, a contradiction. Thus,  $\Sigma_{d}(\mathcal{A})$ is the unique class such that the $d_1\le d_2\le d_3$.

It is clear that cyclic permutations of $d_1,d_2,d_3$ correspond to rotations  by $\pm2\pi/3$ of the solution to Fermat--Steiner problem for $\cA$. Such rotations of a solution from the class $\Sigma_{d}(\mathcal{A})$ takes it onto another solution which belongs to the class with permuted $d_i$. Therefore, there are exactly three classes $\Sigma_{d}(\mathcal{A})$, that completes the proof of the last Item of Theorem.

Steiner compact $K_0$ is minimal element in its class $\Sigma_{d}(\mathcal{A})$, because any proper compact subset of $K_0$ consists of a single point, but $S_\cA\bigl(\{x\}\bigr)\ge 3$ for any $x\in\R^2$. To prove the latter inequality notice that $d_H\bigl(\{x\},A_i\bigr)\ge|xa_i|$, and hence, $S_\cA\bigl(\{x\}\bigr)\ge |xa_1|+|xa_2|+|xa_3|$. In its turn, $|xa_1|+|xa_2|+|xa_3|$ is not less than the length of a shortest tree for the triangle $a_1a_2a_3$ that is equal to $3$. The set $K_{d}(\mathcal{A})$ a maximal Steiner compact in accordance with Assertion~\ref{ass:uniqueness_maximal_compact}. Thus, Items~\ref{item:min_com} and~\ref{item:max_com} of Theorem are proved.
\end{proof}

\end{document}